\documentclass[letterpaper,11pt,twoside]{article}

\usepackage[top=1in,bottom=1in,left=1in,right=1in,includehead]{geometry}

\usepackage{amsmath,amsthm,amsfonts,amssymb}

\usepackage[sc]{mathpazo}
\usepackage[T1]{fontenc}
\usepackage{mathrsfs} 

\usepackage{graphicx,bbm,tikz}
\usetikzlibrary{shapes}

\usepackage{enumitem}
\setenumerate{label=(\alph*)}

\usepackage{authblk}

\newtheorem{theorem}{Theorem}[section]
\newtheorem*{theorem*}{Theorem}
\newtheorem{lemma}[theorem]{Lemma}
\newtheorem*{lemma*}{Lemma}

\theoremstyle{definition}

\newtheorem*{definition*}{Definition}

\newtheorem*{proposition*}{Proposition}

\newtheorem*{example*}{Example}

\newtheorem*{observation*}{Observation}

\newtheorem*{claim*}{Claim}
\newtheorem{remark}[theorem]{Remark}
\newtheorem*{remark*}{Remark}

\numberwithin{equation}{section}

\newcommand{\bN}{\mathbb{N}}

\newcommand{\bR}{\mathbb{R}}

\newcommand{\cC}{\mathcal{C}}

\newcommand{\cN}{\mathcal{N}}
\newcommand{\cO}{\mathcal{O}}

\newcommand{\mJ}{\mathsf{J}}
\newcommand{\mK}{\mathsf{K}}
\newcommand{\mL}{\mathsf{L}}

\newcommand{\E}{\operatorname{\mathbb{E}}}
\newcommand{\Prob}{\operatorname{\mathbb{P}}}
\newcommand{\Var}{\operatorname{Var}}

\newcommand{\parti}{n}

\DeclareRobustCommand{\stirlingii}{\genfrac{\{}{\}}{0pt}{}}

\title{Central limit theorem for peaks of a random permutation in a fixed conjugacy class of $S_n$}
\makeatletter
\let\sparetitle\@title
\makeatother
\author[1]{Jason Fulman\thanks{fulman@usc.edu}}
\author[1]{Gene B. Kim\thanks{genebkim@usc.edu}}
\author[2]{Sangchul Lee\thanks{sos440@math.ucla.edu}}
\affil[1]{University of Southern California}
\affil[2]{University of California, Los Angeles}

\usepackage{fancyhdr}
\begin{document}

\maketitle

\begin{abstract}
    The number of peaks of a random permutation is known to be asymptotically normal. We give a new proof of this and prove a central limit theorem for the distribution of peaks in a fixed conjugacy class of the symmetric group. Our technique is to apply ``analytic combinatorics'' to study a complicated but exact generating function for peaks in a given conjugacy class.
\end{abstract}

\section{Introduction}

We say that a permutation on $n$ symbols has a descent at position $i$ if $\pi(i) > \pi(i+1)$, and we let $d(\pi)$ denote the number of descents of $\pi$. For example. the permutation $1\underline{4}32\underline{6}5$ has descents at positions $2$ and $5$, and has $d(\pi)=2$. Descents appear in numerous parts of mathematics. For examples, see Knuth \cite{Knuth} for connections of descents with the theory of sorting and the theory of runs in permutations and see Bayer and Diaconis \cite{BD} for applications of descents to card shuffling. The number $A(n,k)$ of permutations on $n$ symbols with $k$ descents is called an Eulerian number, and there is an entire book devoted to their study \cite{Pet}.

It is well known that the distribution of descents is asymptotically normal with mean $(n-1)/2$ and variance $(n+1)/12$. There are many proofs of this:
\begin{enumerate}
    \item Pitman \cite{Pit} uses real-rootedness of the Eulerian polynomials
    \begin{align*}
        A_n(t) = \sum_{\pi \in S_n} t^{d(\pi)+1}
    \end{align*}
    
    \item David and Barton \cite{DB} use the method of moments.
    
    \item Tanny \cite{Tan} uses the fact that if $U_1,\cdots,U_n$ are independent uniform $[0,1]$ random variables, the for all integers $k$,
    \begin{align*}
         \Prob\left(k \leq \sum_{i=1}^n U_i < k+1 \right) = A(n,k)/n!
    \end{align*}
    
    \item Fulman \cite{Fstein} uses Stein's method.
\end{enumerate}

There is also interesting literature on the joint distribution of descents and cycles. Gessel and Reutenauer \cite{GR} use symmetric function theory to enumerate permutations with a given cycle structure and descent set, and Diaconis, McGrath, and Pitman \cite{DMP} interpret this in the context of card shuffling. We regard these exact results as a miracle, and they enable one to write down an exact (but quite complicated) generating function for descents of permutations in a given conjugacy class. These exact generating functions make it possible to prove central limit theorems for the number of descents in fixed conjugacy classes of the symmetric group. Fulman \cite{FulCLT} proved a central limit theorem when the conjugacy classes consist of large cycles. Almost twenty years later, Kim \cite{Kim1} proved a central limit for descents in random fixed point free involutions. Quite recently, Kim and Lee \cite{KimLee} proved a central limit theorem for arbitrary conjugacy classes. These
results would be very difficult to obtain without exact generating functions.

Given the above discussion, it is natural to ask if there are other permutation statistics for which there is exact information about the joint distribution with cycle structure. In their work on casino shuffling machines, Diaconis, Fulman, and Holmes \cite{DFH} discovered that there is a lovely exact generating function for the number of peaks of a permutation enumerated according to cycle structure. Let us describe their result. We say that a permutation $\pi \in S_n$ has a peak at position $1<i<n-1$ if $\pi(i-1)< \pi(i) > \pi(i+1)$, and let $p(\pi)$ be the number of peaks of $\pi$. Thus $\pi = 1 \underline{4} 2 6 \underline{7} 5 3$ has peaks at positions 2 and $5$, so that $p(\pi)=2$. Letting $\lambda$ be a partition of $n$ with $\parti_i$ parts of size $i$, Corollary 3.8 of \cite{DFH} gives that
\begin{align}
    \label{eq:peakgen_cycletype}
    \sum_{\pi \in \cC_{\lambda}}
    \left( \frac{4t}{(1+t)^2} \right)^{p(\pi)+1}
    = 2 \left( \frac{1-t}{1+t} \right)^{n+1} \sum_{a \geq 1} t^a \prod_i
    [x_i^{\parti_i}] \left( \frac{1+x_i}{1-x_i} \right)^{f_{a,i}}.
\end{align}
Here, $\cC_{\lambda}$ denotes the elements of $S_n$ of cycle type $\lambda$, and $[x_i^{\parti_i}] g(x_i)$ denotes the coefficient of $x_i^{\parti_i}$ in the function $g(x_i)$, and
\begin{align*}
    f_{a,i} = \frac{1}{2i} \sum_{\substack{d \mid i \\ d \text{ odd}}} \mu(d) (2a)^{i/d},
\end{align*}
where $\mu$ is the M\"obius function of elementary number theory. (The result of \cite{DFH} actually deals with valleys rather than peaks, but the joint generating function with cycle structure is the same as can be seen by conjugating by the longest permutation $n \cdots 21$). The reader will agree that the generating function (\ref{eq:peakgen_cycletype}) looks hard to deal with (it need not be real-rooted), and our main insight is that we can adapt the methods of Kim and Lee \cite{KimLee} to analyze it.

To close the introduction, we mention that the number of peaks of a permutation is a feature of interest. The paper \cite{DFH} uses peaks to analyze casino shelf-shuffling machines. The number of peaks is classically used as a test of randomness for time series; see Warren and Seneta \cite{WS} and their references, which also include a central limit theorem for the number of peaks for a uniform random permutation. Permutations with no peaks are called unimodal (usually unimodal refers to no valleys but these are equivalent for our purposes), and are of interest in social choice theory through Coombs's ``unfolding hypothesis'' (see Chapter 6 of \cite{Dibook}). They also appear in dynamical systems and magic tricks (see Chapter 5 of \cite{DG}).

Finally, we note that peaks have been widely studied by combinatorialists; see Petersen \cite{Petpeak}, Stembridge \cite{Stempeak}, Nyman \cite{Nypeak}, Schocker \cite{Sch} and a paper of Billey, Burdzy, and
Sagan \cite{BBS}, for a small sample of combinatorial work on peaks.

\subsection{Main results}

To motivate the readers, we first demonstrate a numerical simulation result. Figure 1 is a histogram of peaks of $10^5$ permutations drawn from the conjugacy class $\cC_{2^{250} 4^{125}} \subset S_{1000}$.
\begin{center}
    \begin{minipage}{.7\linewidth}
	    \centering
        \includegraphics[width=.75\linewidth]{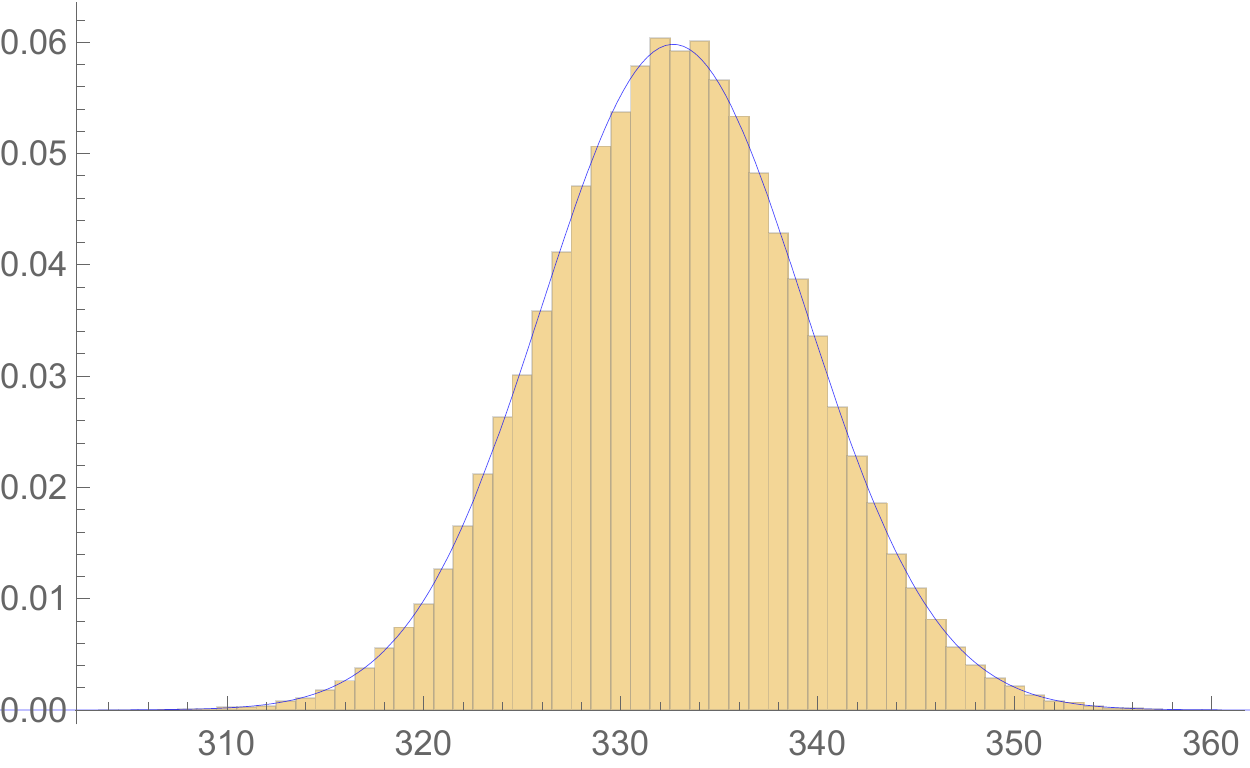}\\
        \small Figure 1. Histogram of peaks of $10^5$ samples drawn from $\cC_{2^{250}4^{125}} \subset S_{1000}$.
    \end{minipage}
\end{center}
The histogram suggests that the peaks of permutations in $\cC_{2^{250} 4^{125}}$ are normally distributed, and indeed, the p.d.f.\ of $\cN \left(\frac{n-2}{3}, \frac{2(n+1)}{45}\right)$ with $n = 1000$ fits very well. This suggests that the behavior of peaks for a particular conjugacy class is mostly the same as that of peaks for $S_n$. This does turn out to be true for conjugacy classes with no fixed points, as the following main theorem states that the asymptotic distribution of peaks in conjugacy classes is normal, where the asymptotic mean and variance depend only on the density of fixed points.

\begin{theorem} \label{mainthm}
        Let $C_n$ be a conjugacy class of $S_n$ for each $n \geq 1$. Denote by $\alpha_1(C_n)$ the fraction of fixed points of each element of $C_n$. Suppose that $\pi_n$ is chosen uniformly at random from $C_n$ and that $\alpha_1(C_n)$ converges to some $\alpha \in [0, 1] $ as $n\to\infty$. Then, as $n\to\infty$,
        \begin{align*}
            \frac{p(\pi_n) - \frac{1-\alpha_1(C_n)^3}{3}n}{\sqrt{n}}
            \quad \text{converges in distribution to} \quad
            \cN\left(0, \tfrac{2}{45} + \tfrac{1}{9}\alpha^3 - \tfrac{3}{5}\alpha^5 + \tfrac{4}{9}\alpha^6 \right).
        \end{align*}
\end{theorem}
    
Our main strategy is to adopt the modified Curtiss' theorem from~\cite{KimLee}, which relates convergence in distribution of random variables to the pointwise convergence of their moment generating functions on an open set. In this regard, the main theorem is a direct consequence of the following technical theorem:
    
\begin{theorem} \label{maintechthm}
    For each $s > 0$, there exists a universal constant $C = C(s) > 0$, depending only on $s$, such that the following is true: Let $\cC_{\lambda} \subseteq S_n$ be the conjugacy class of cycle type $\lambda = 1^{\parti_1} 2^{\parti_2} \cdots$ and $\pi$ be chosen uniformly at random from $\cC_{\lambda}$. Denote by $\alpha_1 = \parti_1 / n$ the density of fixed points. Then,
    \begin{align*}
        \E \left[ e^{-s p(\pi)/\sqrt{n}} \right]
        = \exp\left\{ -\frac{1-\alpha_1^3}{3}s\sqrt{n} + \left( \frac{1}{45} + \frac{\alpha_1^3}{18} - \frac{3\alpha_1^5}{10} + \frac{2\alpha_1^6}{9} \right)s^2 + E_{\lambda,s} \right\},
    \end{align*}
    where $|E_{\lambda,s}| \leq Cn^{-1/4}$.
\end{theorem}

This theorem is interesting in its own right, because the uniform estimate allows us to readily extend the scope of the main theorem to a more general class of sequences $(C_n)$. More precisely, the statement of Theorem \ref{mainthm} readily extends to the case where each $C_n$ is simply a conjugacy-invariant subset of $S_n$ such that every element of $C_n$ has the same number of fixed points. For example, if we consider the set of all elements of $S_n$ with zero fixed points, we would obtain a central limit theorem for peaks of derangements.

\section{Central limit theorem for peaks of a random permutation in $S_n$}

Denoting the peak generating function by
\begin{align*}
    W_n(t) = \sum_{\pi \in S_n} t^{p(\pi) + 1},
\end{align*}
it is well known \cite[p779]{Stempeak} that $A_n(t)$ and $W_n(t)$ are related by the identity
\begin{align}
	\label{eq:peakgen}
	W_n\left(\frac{4t}{(1+t)^2}\right) = \left(\frac{2}{1+t}\right)^{n+1}A_n(t).
\end{align}
Our aim in this section is to identify the asymptotic distribution of peaks of a random permutation in $S_n$ using \eqref{eq:peakgen}.

\subsection{Computing mean and variance of peaks in $S_n$}

We begin by calculating the derivatives of $A_n(t)$ at $1$ up to the fourth order.
\begin{lemma}
    \label{lemma:eulerian_derivatives}
    We have
    \begin{align*}
    	A_n^{(0)}(1) &= n! , \\
    	A_n^{(1)}(1) &= n! \cdot \frac{n+1}{2} \mathbf{1}_{\{ n \geq 1\}}, \\
    	A_n^{(2)}(1) &= n! \cdot \frac{3n^2 + n - 2}{12} \mathbf{1}_{\{ n \geq 2 \}}, \\
    	A_n^{(3)}(1) &= n! \cdot \frac{n^3 - 2n^2 - n + 2}{8} \mathbf{1}_{\{ n \geq 3\}}, \text{ and} \\
    	A_n^{(4)}(1) &= n! \cdot \frac{15n^4 - 90n^3 + 125n^2 + 78n - 152}{240} \mathbf{1}_{\{ n \geq 4\}}.
    \end{align*}
\end{lemma}

\begin{proof}
    It is well known that the Eulerian polynomials satisfy the identity
    \begin{align*}
        A_n(t) = (1 - t)^{n+1} \sum_{a\geq 1} a^n t^a.
    \end{align*}
    Recall that the Stirling numbers of the second kind $\stirlingii{n}{k}$ count the number of partitions of an $n$-element set into $k$ blocks. Plugging the expansion $a^n = \sum_{k=0}^{n} \stirlingii{n}{k} \frac{a!}{(a-k)!}$ into the expression above, we see that
    \begin{align*}
        A_n(t)
        &= (1-t)^{n+1} \sum_{a=0}^{\infty} \left( \sum_{k=0}^{n} \stirlingii{n}{k} \frac{a!}{(a-k)!} \right) t^a \\
        &= (1-t)^{n+1} \sum_{k=0}^{n} \stirlingii{n}{k} \frac{k!t^k}{(1-t)^{k+1}} \\
        &= \sum_{k=0}^{n} k! \stirlingii{n}{k} t^k (1 - t)^{n-k} \\
        &= \sum_{k=0}^{n} (n-k)! \stirlingii{n}{n-k} t^{n-k} (1 - t)^{k}.
    \end{align*}
    Now one can compute $A_n^{(p)}(1)$ by plugging the above identity into $A_n^{(p)}(1) = p![s^p]A_n(1+s)$. More specifically, if $p > n$, then $A_n(1+s)$ has degree $n$, and so, $A_n^{(p)}(1) = 0$. If $p \leq n$, then
    \begin{align*}
        A_n^{(p)}(1)
         = p![s^p]A_n(1+s)
        &= p![s^p]\sum_{k=0}^{n} (-1)^k (n-k)! \stirlingii{n}{n-k} (1+s)^{n-k} s^{k} \\
        &= p! \sum_{k=0}^{p} (-1)^k (n-k)! \stirlingii{n}{n-k} \binom{n-k}{p-k}.
    \end{align*}
    For each given $p$, the last sum can be computed by calculating $\stirlingii{n}{n-k}$'s for $k = 0, \cdots, p$. For instance, $\stirlingii{n}{n} = 1$ and $\stirlingii{n}{n-1} = \binom{n}{2}$, and for larger values of $k$, they can be systematically computed by utilizing the relationship between the Stirling numbers of the second kind and Eulerian numbers of the second kind (see equation (6.43) of \cite{GKP}). The $\stirlingii{n}{n-k}$'s relevant to us are
    \begin{align*}
        \stirlingii{n}{n-2} &= 2 \binom{n}{4} + \binom{n+1}{4}, \\
        \stirlingii{n}{n-3} &= 6 \binom{n}{6} + 8 \binom{n+1}{6} + \binom{n+2}{6}, \text{ and}\\
        \stirlingii{n}{n-4} &= 24 \binom{n}{8} + 58 \binom{n+1}{8} + 22 \binom{n+2}{8} + \binom{n+3}{8}.
    \end{align*}
    Plugging these back into the formula for $A_n^{(p)}(1)$ provides the desired lemma.
\end{proof}

Next, \eqref{eq:peakgen} relates $W_n^{(p)}(1)$ to the derivatives of $A_n(t)$ up to order $2p$ evaluated at~$1$. Differentiating both sides of \eqref{eq:peakgen} gives us
\begin{align*}
    - \frac{4(t-1)}{(1+t)^3}W_n'\left(\frac{4t}{(1+t)^2}\right)
    = -\frac{(n+1)2^{n+1}}{(1+t)^{n+2}}A_n(t) + \frac{2^{n+1}}{(1+t)^{n+1}}A_n'(t),
\end{align*}
and by multiplying $- \frac{(1+t)^3}{4(t-1)}$ to both sides and simplifying, we see that
\begin{align*}
	W_n'\left(\frac{4t}{(1+t)^2}\right) = \left(\frac{2}{1+t}\right)^{n-1} \frac{(n+1)A_n(t) - (t+1)A_n'(t)}{t-1}.
\end{align*}
This formula cannot be evaluated directly at $ t = 1$, but we can use L'H\^opital's rule to get
\begin{align*}
	W_n'(1)
	&= \lim_{t \to 1} W_n'\left(\frac{4t}{(1+t)^2}\right)
	 = \lim_{t \to 1} \frac{(n+1)A_n(t) - (t+1)A_n'(t)}{t-1} \\
	&= n A_n'(1) - 2A_n''(1) 
	 = n! \cdot \frac{n+1}{3}, \qquad \text{if } n \geq 2.
\end{align*}
The last step is a consequence of Lemma \ref{lemma:eulerian_derivatives}. The second derivative $W_n''(1)$ can be computed in similar fashion. By differentiating both sides of \eqref{eq:peakgen} twice and simplifying, we obtain an identity relating $W_n''$ to the derivatives of $A_n$:
\begin{align*}
	W_n''\left(\frac{4t}{(1+t)^2}\right) = \left(\frac{2}{1+t}\right)^{n-3} \frac{P_n(t)}{(t-1)^3},
\end{align*}
where $P_n(t)$ is given by
\begin{align*}
    P_n(t) = (n+1)(nt-n+2)A(t) - 2(t+1)(nt-n+1)A'(t) + (t-1)(t+1)^2 A''(t).
\end{align*}
Similarly as before, we find $W_n''(1)$ by using L'H\^opital's rule:
\begin{align*}
	W_n''(1)
	&= \lim_{t \to 1} W_n''\left(\frac{4t}{(1+t)^2}\right)
	= \lim_{t \to 1} \frac{P_n(t)}{(t-1)^3}
	= \frac{P_n^{(3)}(1)}{6} \\
	&= \frac{(3n^2-9n+6) A^{(2)}(1) - (10n-20) A^{(3)}(1) + 8 A^{(4)}(1)}{6} \\
	&= n! \cdot \frac{(5n-8)(n+1)}{45}, \qquad \text{if } n \geq 4,
\end{align*}
where the last step follows from Lemma \ref{lemma:eulerian_derivatives}. Finally, since $W_n'(1) = n! \E [p(\pi) + 1]$ and $W_n''(1) = n! \E [(p(\pi) + 1)p(\pi)]$, we have
\begin{align*}
	\E[p(\pi)] = \frac{W_n'(1)}{n!} - 1 = \frac{n-2}{3} \qquad \text{if } n \geq 2,
\end{align*}
and
\begin{align*}
	\Var(p(\pi)) = \frac{W_n''(1)}{n!} + \frac{W_n'(1)}{n!} - \left( \frac{W_n'(1)}{n!} \right)^2 = \frac{2(n+1)}{45} \qquad \text{if } n \geq 4.
\end{align*}

At this point, it is worth noting \eqref{eq:peakgen} implies that, like $A_n(t)$, $W_n(t)$ has only real roots, and so, by Harper's method~\cite{Harper}, we can obtain a central limit theorem for peaks of a random permutation in $S_n$. In the upcoming section, we give a new proof of this central limit theorem by using analytic combinatorics and will go further to prove a central limit theorem for peaks in arbitrary conjugacy classes of $S_n$, where the mean and variance depend only on the density of fixed points in the conjugacy classes.

\subsection{Establishing the asymptotic normality of peaks in $S_n$}

Kim and Lee~\cite{KimLee} proved the following modification of Curtiss' theorem:
\begin{theorem}
    \label{theorem:curtiss}
    Let $X_n$ be random vectors in $\mathbb{R}^d$ for each $n \in \mathbb{N} \cup \left\{ \infty \right\}$ and $M_{X_n}(s) = \E[e^{sX_n}]$ be the moment generating function (m.g.f\@.) of $X_n$. Suppose that there is a non-empty open subset $U \subseteq \mathbb{R}^d$ such that $\lim_{n \to \infty} M_{X_n}(s) = M_{X_\infty}(s)$ for all $s \in U$. Then, $X_n$ converges in distribution to $X_\infty$.
\end{theorem}

This theorem will be used in this subsection to prove a central limit theorem about peaks of permutations chosen, uniformly at random, from $S_n$, and in section 3 to prove an analogous theorem about peaks of permutations chosen, uniformly at random, from arbitrary conjugacy classes, where the asymptotic mean and variance are functions of only $\alpha$, the density of fixed points in the conjugacy classes.

\begin{theorem}
    \label{thm:clt_peaks_Sn}
    Let $\pi_n$ be chosen uniformly at random from $S_n$. Then $p(\pi_n)$ is asymptotically normal with mean $\frac{n-2}{3}$ and variance $\frac{2(n+1)}{45}$. More precisely, as $n \to \infty$,
    \begin{align*}
        \frac{p(\pi_n) - \tfrac{n-2}{3}}{\sqrt{n}}
        \quad \text{converges in distribution to} \quad \cN\left(0, \tfrac{2}{45}\right).
    \end{align*}
\end{theorem}

\begin{proof}
    Let $X_n = \left( p(\pi_n) - \frac{n-2}{3} \right) / \sqrt{n}$ denote the normalized peaks. In view of Theorem \ref{theorem:curtiss}, it suffices to show that $M_{X_n}(s)$ converges pointwise to the m.g.f.\ of $\cN \left(0, \frac{2}{45}\right)$ on some open interval. Let $0 < t < 1$. By a simple comparison, it follows that
    \begin{align*}
    	t \cdot \frac{n!}{\log^{n+1}(1/t)}
    	=  \int_{0}^{\infty} a^{n} t^{a+1} \, \mathrm{d}a
    	\leq \sum_{a\geq 1} a^{n} t^{a}
    	\leq \int_{0}^{\infty} a^{n} t^{a-1} \, \mathrm{d}a
    	= \frac{1}{t} \cdot \frac{n!}{\log^{n+1}(1/t)}.
    \end{align*}
    Plugging this into \eqref{eq:peakgen}, we obtain
    \begin{align*}
        \frac{1}{n!} W_n\left(\frac{4t}{(1+t)^2}\right)
        = \frac{1}{n!} \left(\frac{2(1 - t)}{1+t}\right)^{n+1} \left( \sum_{a\geq 1} a^n t^a \right)
    	= e^{\mathcal{O}(\log t)} \left( \frac{2(1-t)}{(1+t)\log(1/t)} \right)^{n+1}.
    \end{align*}
    Now, fix $s > 0$ and choose $t$ as the unique solution of $\frac{4t}{(1+t)^2} = e^{-s/\sqrt{n}}$ in the range $(0, 1)$, which is given by
    \begin{align}
        \label{eq:t_asymp}
        t
        = \frac{1 - \sqrt{1 - e^{-s/\sqrt{n}}}}{1 + \sqrt{1 - e^{-s/\sqrt{n}}}}
        = 1 - \frac{2s^{1/2}}{n^{1/4}} + \frac{2 s}{n^{1/2}} - \frac{3 s^{3/2}}{2n^{3/4}} + \frac{s^2}{n} + \cO\left(n^{-5/4}\right),
    \end{align}
    where the implicit bound of the error term depends only on $s$. From this expansion, we have both $\log(t) = \cO\left(n^{-1/4}\right)$ and $\log\left(\frac{2(1-t)}{(1+t)\log(1/t)} \right) = -\frac{s}{3\sqrt{n}} + \frac{s^2}{45n} + \mathcal{O}(n^{-5/4}) $. Plugging these into $M_{X_n}(s)$, we see that
    \begin{align*}
        M_{X_n}(s)
        = \frac{1}{n!} W_n\left(e^{-s/\sqrt{n}}\right)  e^{\frac{n+1}{3\sqrt{n}}s}
        = e^{\frac{s^2}{45} + \cO\left(n^{-1/4}\right)}.
    \end{align*}
    The desired conclusion follows since $e^{s^2/45}$ is the m.g.f.\ of the $\cN\left( 0, \frac{2}{45} \right)$.
\end{proof}

\section{Central limit theorem for peaks of a random permutation in a fixed conjugacy class of $S_n$}

Let $\cC_{\lambda}$ denote the set of all permutations of $S_n$ of cycle type $\lambda = 1^{\parti_1} 2^{\parti_2} \cdots $ of $n$. Recall that the peak generating function over $\cC_{\lambda}$ has an explicit formula \eqref{eq:peakgen_cycletype}, which involves the quantity~$f_{a,i}$ defined in the introduction. Along the proof of the main theorem, it is important to know a precise estimation of $f_{a,i}$. Define $g_{a,i}$ by the following relation
\begin{align*}
    f_{a,i} = \frac{(2a)^i}{2i} g_{a,i}.
\end{align*}
The main reason for introducing $g_{a,i}$ is that $f_{a,i}$ is expected to behave much like $(2a)^i/(2i)$, and so, it is necessary to study the relative difference and produce a precise estimate for the difference. The following lemma serves this purpose.

\begin{lemma}
    \label{lemma:fai_est}
	There exists a universal constant $c_1 > 0$ such that
	\begin{align*}
		e^{-c_1 (2a)^{-2i/3}}
		\leq g_{a,i}
		\leq e^{c_1 (2a)^{-2i/3}}
	\end{align*}
	for all $a \geq 1$ and $ i \geq 1$. Consequently, we have $e^{-(c_1/4) / a^2} \leq g_{a,i} \leq e^{(c_1/4) / a^2}$.
\end{lemma}

Although the intermediate step of the proof will show that the explicit choice $c_1 = 4$ works, we prefer to leave it as a named constant. This is because its value is not important for the argument and its presence will clarify the way we utilize this lemma.

\begin{proof}
	Recall that $f_{a,i} = \frac{1}{2i} \sum \mu(d) (2a)^{i/d}$, where the sum is over $d$, the positive odd divisors of $i$. From this, we see that $g_{a,i} = 1$ when $i$ is either $1$ or $2$, and so, it suffices to assume that $i \geq 3$. For such $i \geq 3$,
	\begin{align*}
		(2a)^i \left| g_{a,i} - 1 \right|
		\leq \sum_{\substack{ d \mid i \\ d \text{ odd, } d \neq 1 }} (2a)^{i/d}
		\leq \sum_{k=1}^{\lfloor i/3 \rfloor} (2a)^k
		= \frac{2a}{2a-1} \left( (2a)^{\lfloor i/3 \rfloor} - 1 \right)
		\leq 2 (2a)^{i/3}.
	\end{align*}
	Rearranging, it follows that
	\begin{align*}
	    1 - 2(2a)^{-2i/3} \leq g_{a,i} \leq 1 + 2(2a)^{-2i/3}.
	\end{align*}
	Since $a \geq 1$ and $i \geq 3$, we have $2(2a)^{-2i/3} \leq \frac{1}{2}$. Then, applying the inequalities $e^{-2x} \leq 1-x$ and $ 1+x \leq e^{2x}$, which are valid for $0 \leq x \leq \frac{1}{2}$, proves the claim with the choice $c_1 = 4$. The remaining assertion is a simple consequence of the fact that $(2a)^{-2i/3} \leq a^{-2}$ for $i \geq 3$.
\end{proof}

\begin{remark}
    The quantity $f_{a,i}$ is a positive integer. In the special case when $a$ is a power of $2$, this follows from Lemma 1.3.16 of \cite{FNP}, which enumerates monic, irreducible, self-conjugate polynomials of degree $2i$ over a finite field of size $2a$.
    
    For general $a$, the quantity $f_{a,i}$ enumerates what Victor Reiner calls ``nowhere-zero primitive twisted necklaces'' with values in
    \begin{align*}
        A = \{+1,-1,+2,-2,\cdots,+a,-a \}
    \end{align*}
    having $i$ entries. To define this notion, let the cyclic group $C_{2i}$ act on $i$-tuples of words $(b_1,\cdots,b_i)$ where the $b_k$'s take values in $A$, and the generator of $C_{2i}$ acts by
    \begin{align*}
        g(b_1,\cdots,b_i) = (b_2,\cdots,b_i,-b_1).
    \end{align*}
    An orbit $P$ of this action is called a twisted necklace, and $P$ primitive means that the $C_{2i}$ action is free (i.e. no non-trivial group element fixes any vector in the orbit $P$). Arguing as in the proof of Theorem 4.2 of \cite{Reiner} shows that $f_{a,i}$ does indeed enumerate nowhere-zero primitive twisted necklaces. We thank Victor Reiner for this observation.
\end{remark}

\subsection{Heuristics and main idea}

We begin by focusing on the product of coefficients appearing in the formula of the peak generating function \eqref{eq:peakgen_cycletype}. More specifically, we seek to find a formula of each coefficient that is more manageable for estimation. Applying the generalized binomial theorem to expand the function, we get
\begin{align}
	[x_{i}^{\parti_i}] \left( \frac{1+x_i}{1-x_i} \right)^{f_{a,i}}
	&= [x_{i}^{\parti_i}] \left( \left( 1+x_i \right)^{f_{a,i}} \left( 1-x_i \right)^{-f_{a,i}} \right) \nonumber \\
	&= \sum_{k=0}^{\infty} \binom{f_{a,i}}{k} \binom{f_{a,i} - 1 + \parti_i - k}{f_{a,i} - 1}
	 = \frac{(2f_{a,i})^{\parti_i}}{\parti_i!} \mK_{a,i}, \label{eq:kai}
\end{align}
where $\mK_{a,i}$ is defined by
\begin{align*}
    \mK_{a,i}
    = \sum_{\nu = 0}^{\parti_i} \frac{1}{2^{\parti_i}} \binom{\parti_i}{\nu} \frac{(f_{a,i} - \nu + \parti_i - 1)!}{(f_{a,i} - \nu)! f_{a,i}^{\parti_i - 1}}.
\end{align*}

To apply \eqref{eq:kai}, note that the term $t^a \prod_i [x_{i}^{\parti_i}] \left( (1+x_i)/(1-x_i) \right)^{f_{a,i}}$ in \eqref{eq:peakgen_cycletype} appears to contribute to the sum meaningfully only when $a$ is comparable to $n^{5/4}$. Also, the $\mK_{a,i}$'s are approximately $1$ if $f_{a,i}$ is considerably larger than $\parti_i$. If all these observations get along, one may argue heuristically that
\begin{align*}
	\frac{1}{|\cC_{\lambda}|}\sum_{\pi \in \cC_{\lambda}} \left( \frac{4t}{(1+t)^2} \right)^{p(\pi) + 1}
	&\stackrel{?}{\approx} \left( \frac{\prod_i \parti_i! i^{\parti_i}}{n!} \right) \cdot 2 \left( \frac{1-t}{1+t} \right)^{n+1}
	\int_{0}^{\infty} t^{x} \prod_{i} \frac{\left( (2x)^{i} / i \right)^{\parti_i}}{\parti_i!} \, \mathrm{d}x \\
	&= \frac{1}{n!} \left( \frac{2(1-t)}{1+t} \right)^{n+1} \int_{0}^{\infty} t^x x^n \, \mathrm{d}x \\
	&= \left( \frac{2(1-t)}{(1+t)\log(1/t)} \right)^{n+1}.
\end{align*}

The final result is the same as what appears in the proof of the asymptotic normality of peaks over $S_n$. This leads to a naive guess that the peaks over $\cC_{\lambda}$ have asymptotically the same normal distribution as the peaks over $S_n$. Of course, we must test the validity of this claim. One main concern is that the alleged asymptotic behavior of \eqref{eq:kai} may not be valid for small $i$'s. Such phenomenon is already observed in the case of descents \cite{KimLee}, where the asymptotic distribution of descents for a fixed cycle type is parametrized by the density of fixed points. And indeed, we will find that corrections are also needed for the peak distribution due to the presence of fixed points. In summary, we need to
\begin{itemize}
    \item precisely control error terms appearing in various approximations, and
    
    \item investigate how the presence of fixed points affects the asymptotic formula for the peak generating function.
\end{itemize}

From this point forward, let $s > 0$ be a fixed positive real number. Then, $t$ is chosen as in \eqref{eq:t_asymp}, which is the unique solution of $\frac{4t}{(1+t)^2} = e^{-s/\sqrt{n}}$ in the interval $(0, 1)$. As the first step of rigorization, we mimic the heuristic computation without using approximations. Applying \eqref{eq:kai} to the peak generating function \eqref{eq:peakgen_cycletype}, we get
\begin{align*}
	\frac{1}{|\cC_{\lambda}|}\sum_{\pi \in \cC_{\lambda}} e^{-\frac{s}{\sqrt{n}}(p(\pi) + 1)}
	&= \frac{2}{n!} \left( \frac{1-t}{1+t} \right)^{n+1}
	\sum_{a \geq 1} t^{a} \prod_{1 \leq i \leq n} \parti_i! i^{\parti_i}  [x_{i}^{\parti_i}] \left( \frac{1+x_i}{1-x_i} \right)^{f_{a,i}} \\
	&= \frac{2}{n!} \left( \frac{1-t}{1+t} \right)^{n+1}
	\sum_{a \geq 1} t^{a} \prod_{1 \leq i \leq n}  (2a)^{i\parti_i} g_{a,i}^{\parti_i} \mK_{a,i} \\
	&= \left( \frac{2(1-t)}{(1+t)\log(1/t)} \right)^{n+1}
	\left[ \frac{\log^{n+1}(1/t)}{n!} \sum_{a \geq 1} a^n t^{a} \prod_{1 \leq i \leq n} g_{a,i}^{\parti_i} \mK_{a,i} \right].
\end{align*}
For the sake of conciseness, define $\mL_{\bullet}$ by
\begin{align*}
    \mL_{A} := \frac{\log^{n+1}(1/t)}{n!} \sum_{a \in A \cap \bN} a^n t^{a} \prod_{1 \leq i \leq n} g_{a,i}^{\parti_i} \mK_{a,i}
\end{align*}
for all $A \subseteq \mathbb{R}$. Then, the above computation simplifies to
\begin{align}
    \label{eq:peakgen_cycletype_simple}
    \frac{1}{|\cC_{\lambda}|}\sum_{\pi \in \cC_{\lambda}} e^{-\frac{s}{\sqrt{n}}(p(\pi) + 1)}
    = \left( \frac{2(1-t)}{(1+t)\log(1/t)} \right)^{n+1} \mL_{[1,\infty)}.
\end{align}
As in the heuristic computation, $\mL_{\bullet}$ will be approximated by its integral analogue. In doing so, it is convenient to split the sum into two parts at a certain threshold. The primary reason is that the aforementioned approximation tends to fail for small $a$, and so, such case deserves to be handled separately. To describe this threshold, let
\begin{align}
    \label{eq:delta}
    \delta_0 = \left[ \sup_{n \geq 1} \left( n^{1/4} \log(1/t) e^{(c_1/4)+1} \right) \right]^{-1}
\end{align}
and fix any $\delta \in (0, \delta_0)$. In view of \eqref{eq:t_asymp}, $\log(1/t) = 2\sqrt{s}n^{-1/4} + \cO(n^{-3/4})$ for large~$n$. This guarantees that $\delta_0$ is away from $0$, and so, the choice of $\delta$ does make sense. Then, the sum~$\mL_{[1,\infty)}$ will be split into $\mL_{[1, \delta n^{5/4}]} + \mL_{(\delta n^{5/4},\infty)}$, and we will call the former term the \emph{small range} and the latter term the \emph{large range}.

\subsection{Estimation of small range}

We will focus on the range $a \leq \delta n^{5/4}$, where $\delta$ will be chosen from $(0, \delta_0)$. The main goal in this section is to show that the contribution arising from this range is negligible. The precise statement is as follows.

\begin{lemma}
    \label{lemma:small_range}
    For each $\delta \in (0, \delta_0)$ and $\rho \in (\delta/\delta_0, 1)$, there exists a constant $c_3 = c_3(\delta, \rho) > 0$, depending only on $\delta$ and $\rho$, such that
    \begin{align*}
        \mL_{[0, \delta n^{5/4}]} \leq c_3 \rho^{n+1}.
    \end{align*}
\end{lemma}

We begin by producing a simple upper bound for the product of the $\mK_{a,i}$'s.

\begin{lemma}
    \label{lemma:kai_asymp_small}
    Let $\delta > 0$. Then, there exists a constant $c_2 = c_2(\delta) > 0$, depending only on $\delta$, such that
    \begin{align}
        \prod_{1 \leq i \leq n} \mK_{a,i} \leq \left( \frac{\delta n^{5/4}}{a} \right)^{n} e^{c_2 n^{3/4}}
    \end{align}
    whenever $a \leq \delta n^{5/4}$ holds.
\end{lemma}

\begin{proof}
    Assume that $a \leq \delta n^{5/4}$. If $0 \leq \nu \leq \parti_i$, then
    \begin{align*}
        \frac{(f_{a,i} - \nu + \parti_i - 1)!}{(f_{a,i} - \nu)! f_{a,i}^{\parti_i - 1}}
        &= \prod_{k=1}^{\parti_i - 1} \left( 1 + \frac{k - \nu}{f_{a,i}} \right)
        \leq \left( 1 + \frac{\parti_i}{f_{a,i}} \right)^{\parti_i}.
    \end{align*}
    Plugging this to the definition of $\mK_{a,i}$, we obtain $\mK_{a,i} \leq \left(1 + (\parti_i / f_{a,i}) \right)^{\parti_i}$. This bound will be further simplified depending on whether $i = 1$ or $i \geq 2$. For the sake of brevity, we write $r = \delta n^{5/4} / a$. By assumption, we have $r \geq 1$. Now, when $i = 1$, plug $f_{a,1} = a$ and proceed as
    \begin{align*}
    	\mK_{a,1}
    	\leq \left( 1 + \frac{\parti_1}{a} \right)^{\parti_1}
    	\leq r^{\parti_1} \left( 1 + \frac{\parti_1}{ra} \right)^{\parti_1}
    	\leq r^{\parti_1} e^{\parti_1^2 / ra }
    	\leq r^{\parti_1} e^{(1/\delta) n^{3/4}}.
    \end{align*}
    In the third and fourth steps, inequalities $1+x \leq e^x$ and $\parti_1 \leq n$ are utilized, respectively. Likewise, when~$i \geq 2$, we apply Lemma \ref{lemma:fai_est} and proceed as in the previous case to get
    \begin{align*}
    	\mK_{a,i}
    	\leq \left( 1 + 2e^{c_1} \frac{i\parti_i}{(2a)^i} \right)^{\parti_i}
        \leq r^{i\parti_i} \left( 1 + 2e^{c_1} \frac{i\parti_i}{(2ra)^i} \right)^{\parti_i}
    	\leq r^{i\parti_i} e^{2e^{c_1} i\parti_i^2 / (2ra)^i }
    	\leq r^{i\parti_i} e^{(e^{c_1}/\delta^2) i\parti_i / n^{3/2} }.
    \end{align*}
    In the third step, the obvious inequality $i\parti_i \leq n$ is used. Combining altogether and utilizing the identity $\sum_{i \geq 2} i \parti_i = n - \parti_1$, we see that
    \begin{align*}
    	\prod_{1 \leq i \leq n} \mK_{a,i}
    	\leq \left( r^{\parti_1} e^{(1/\delta) n^{3/4}} \right) \left( r e^{(e^{c_1}/\delta^2) / n^{3/2} } \right)^{n-\parti_1}
    	\leq r^{n} e^{c_2 n^{3/4}},
    \end{align*}
    where $c_2$ can be chosen as $c_2 = (1/\delta) + (e^{c_1}/\delta^2)$.
\end{proof}

\begin{proof}[Proof of Lemma \ref{lemma:small_range}]
    By Lemmas \ref{lemma:fai_est} and \ref{lemma:kai_asymp_small}, we see that
    \begin{align*}
    	\mL_{[0, \delta n^{5/4}]}
    	&\leq \frac{\log^{n+1}(1/t)}{n!} \sum_{1 \leq a \leq \delta n^{5/4}} (\delta n^{5/4})^{n} t^{a} e^{c_2 n^{3/4}} e^{(c_1/4) n/ a^2} \\
    	&\leq \frac{\log^{n+1}(1/t)}{n!} (\delta n^{5/4})^{n+1} e^{c_2 n^{3/4}} e^{(c_1/4) n}
    \end{align*}
    Here, the last step follows by taking the union bound together with the fact that $t^a \leq 1$. Now, by the definition of $\delta_0$, we have $n^{1/4}\log(1/t)e^{(c_1/4)+1} \leq 1/\delta_0$. Moreover, a quantitative form of the Stirling's formula \cite{Robbins} tells us that $n! \geq \sqrt{2\pi} n^{n+1/2}e^{-n}$, and so,
    \begin{align*}
        \mL_{[0, \delta n^{5/4}]}
    	&\leq \frac{1}{(2\pi)^{1/2}n^{n+1/2}e^{-n}} \left( \frac{1}{\delta_0 n^{1/4} e^{(c_1/4)+1}} \right)^{n+1} (\delta n^{5/4})^{n+1} e^{c_2 n^{3/4}} e^{(c_1/4) n} \\
    	&= \rho^{n+1} \cdot \left(\frac{\delta}{\delta_0 \rho} \right)^{n+1} \frac{n^{1/2} e^{c_2 n^{3/4}}}{(2\pi)^{1/2} e^{(c_1/4)+1}}.
    \end{align*}
    If $\rho \in (\delta/\delta_0, 1)$, then the factor $(\delta/\rho\delta_0)^{n+1} n^{1/2}e^{c_2 n^{3/4}} $ is bounded, and hence, the claim follows.
\end{proof}

\subsection{Estimation of large range}

We now turn our attention to the range $a > \delta n^{5/4}$, where we recall that $\delta > 0$ is a fixed number chosen to satisfy \eqref{eq:delta}. We begin by proving the following lemma, which resolves the contribution of the~$\mK_{a,i}$'s for~$i \geq 2$.

\begin{lemma}
    \label{lemma:kai_asymp_large}
    There exists a universal constant $c_4 > 0$ such that
    \begin{align*}
    	e^{ -c_4 n^2 / a^2}
    	\leq \prod_{i \geq 2} \mK_{a,i}
    	\leq e^{ c_4 n^2 / a^2}
    \end{align*}
    whenever $a \geq \delta n^{5/4} \geq e^{c_1} n$. Here, $c_1$ is chosen as in Lemma \ref{lemma:fai_est}.
\end{lemma}

\begin{proof}
    Assume that $a \geq \delta n^{5/4} \geq e^{c_1} n$. When $i \geq 2$, Lemma \ref{lemma:fai_est} gives us that $f_{a,i} \geq e^{-c_1}\frac{(2a)^{2}}{2i} \geq \frac{2na}{i} \geq 2e^{c_1}n \geq 2\parti_i$. Now, letting $0 \leq \nu \leq \parti_1$, we have, as in the beginning of the proof of Lemma \ref{lemma:kai_asymp_small},
    \begin{align*}
    	\left( 1 - \frac{\parti_i}{f_{a,i}} \right)^{\parti_i}
    	\leq \mK_{a,i}
    	\leq \left( 1 + \frac{\parti_i}{f_{a,i}} \right)^{\parti_i}.
    \end{align*}
    Since $\frac{\parti_i}{f_{a,i}} \leq \frac{1}{2}$, we may apply inequalities $-2x \leq \log(1-x) $ and $ \log(1+x) \leq 2x$, which are valid for $0 \leq x \leq \frac{1}{2}$, to further simplify the above bounds, which results in
    \begin{align*}
    	-e^{c_1} \frac{i \parti_i^2}{a^2}
    	\leq -\frac{2\parti_i^2}{f_{a,i}}
    	\leq \log(\mK_{a,i})
    	\leq \frac{2\parti_i^2}{f_{a,i}}
    	\leq e^{c_1} \frac{i \parti_i^2}{a^2}.
    \end{align*}
    Finally, by summing this inequality for $i = 2, \cdots, n$ and utilizing the bound $\sum_{i} i \parti_i^2 \leq n^2$, the desired conclusion follows with $c_4 = e^{c_1}$.
\end{proof}

Next, we establish a detailed asymptotic expansion of $\mK_{a,1}$.

\begin{lemma}
    \label{lemma:ka1_asymp_large}
    Let $\delta \in (0, \delta_0)$. Then,
    \begin{align*}
        \mK_{a,1} = \exp\left\{ \frac{\parti_1^3}{12 a^2} - \frac{3 \parti_1^5}{160 a^4} + \cO \left(n^{-1/4}\right) \right\}
    \end{align*}
    holds in the range $a \geq \delta n^{5/4} \geq 2n$. Moreover, the implicit bound of the error term depends only on $s$ and~$\delta$.
\end{lemma}

\begin{proof}
    It is convenient to separate the case of small $\parti_1$ from the general argument. Letting $0 \leq \nu \leq \parti_1$ and using the fact that $1+x = e^{x + \cO(x^2)}$ near $x = 0$, we get
    \begin{align*}
        \frac{(a - \nu + \parti_i - 1)!}{(a - \nu)! a^{\parti_i - 1}}
        = \prod_{k=1}^{\parti_1 - 1} \left( 1 + \frac{k-\nu}{a} \right)
        = \exp\left\{ \sum_{k=1}^{\parti_1 - 1} \left( \frac{k-\nu}{a} + \cO\left( \frac{\parti_1^2}{a^2} \right) \right) \right\}.
    \end{align*}
    So, if $N$ is a random variable having binomial distribution with parameters $\parti_1$ and $\frac{1}{2}$, then
    \begin{align*}
        \mK_{a,1}
        = \E \left[ \frac{(a - N + \parti_i - 1)!}{(a - N)! a^{\parti_i - 1}} \right]
        = e^{\cO(\parti_1^3/a^2)} \E \left[ \exp\left\{ \frac{\parti_1 - 1}{a} \left( \frac{\parti_1}{2} - N \right) \right\} \right]
    \end{align*}
    and
    \begin{align*}
        \E \left[ \exp\left\{ \frac{\parti_1 - 1}{a} \left( \frac{\parti_1}{2} - N \right) \right\} \right]
        = \cosh^{\parti_1} \left( \frac{\parti_1 - 1}{2a} \right)
        = e^{\cO(\parti_1^3/a^2)},
    \end{align*}
    where we utilized the fact that $\cosh(x) = e^{\cO(x^2)}$ near $x = 0$. In particular, if we set $\beta = \frac{3}{4}$ and assume that $\parti_1 \leq n^{\beta}$, then $\parti_1^3/a^2 \leq \delta^{-2} n^{-1/4}$, and so, the conclusion of the lemma holds. Again, we prefer to use the named variable $\beta$ rather than the actual value in order to emphasize how it is employed in each step of the proof.
    
    The previous computation leads our attention to the case $\parti_1 \geq n^{\beta}$ with $\beta = \frac{3}{4}$. In such case, we will write
    \begin{align*}
        \mK_{a,1} = \sum_{\nu=0}^{\parti_1} p(\nu),
        \quad \text{where} \quad
        p(\nu) = \frac{1}{2^{\parti_1}} \binom{\parti_1}{\nu} \frac{(a + \parti_1 - 1 - \nu)!}{(a - \nu)! a^{\parti_1 - 1}}.
    \end{align*}
    We adopt the idea of Laplace's method to estimate $\mK_{a,1}$. That said, we will argue by showing that $p(\nu)$ is approximately a gaussian density. Our goal is to establish a rigorous version of this claim and then draw the desired estimate from it.
    
    We first obtain a global upper bound of $p$. Identify the factorial $n!$ with the gamma function $\Gamma(n+1)$ so that $p(\nu)$ is defined as an analytic function of $\nu$ on $[0, \parti_1]$. It is well known that the second derivative of the log-gamma function satisfies $(\log \Gamma(z+1))'' = \sum_{n=1}^{\infty} (n+z)^{-2}$, and so,
    \begin{align*}
        (\log p(\nu))''
        &= -\sum_{n=1}^{\infty} \left( \frac{1}{(\nu + n)^2} + \frac{1}{(\parti_1 - \nu + n)^2} \right) - \sum_{k=1}^{\parti_1 - 1} \frac{1}{(a-\nu+k)^2} \\
        &\leq -\sum_{n=1}^{\infty} \frac{2}{(\frac{\parti_1}{2} + n)^2}
         \leq - \int_{1}^{\infty} \frac{2}{(\frac{\parti_1}{2} + x)^2} \, \mathrm{d}x
         = - \frac{4}{\parti_1 + 2}.
    \end{align*}
    In particular, $(\log p(\nu))'$ is strictly decreasing on $[0, \parti_1]$. Moreover, there exists a unique solution $\nu = \tilde{\nu}_{0}$ of the equation $\log p(\nu+1) - \log p(\nu) = 0$ on $[0, \parti_1]$, which is explicitly given by
    \begin{align}
        \label{eq:nutilde0_asymp}
        \tilde{\nu}_{0}
        = \frac{a + \parti_1 - 1 - \sqrt{a^2 + \parti_1^2 - 1}}{2}
        = \frac{\parti_1}{2} - \frac{\parti_1^2}{4a} + \frac{\parti_1^4}{16a^3} + \cO\left(1\right).
    \end{align}
    Then, by the mean-value theorem, there exists $\nu_0 \in [\tilde{\nu}_{0},\tilde{\nu}_{0}+1]$ at which $(\log p(\nu))'$ vanishes, and $\nu_0$ is unique by the strict monotonicity. Integrating twice, we get
    \begin{align}
        \label{eq:pnu_upper_bound}
        p(\nu)
        = p(\nu_{0}) \exp\left\{ \int_{\nu_0}^{\nu} (\nu - t) (\log p(t))'' \, \mathrm{d} t \right\}
        \leq p(\nu_{0}) \exp\left\{ -\frac{2}{\parti_1 + 2} (\nu - \nu_{0})^2 \right\}.
    \end{align}
    
    Next, we claim that this upper bound is a correct asymptotic formula for $p(\nu)$, which amounts to providing a lower bound similar to \eqref{eq:pnu_upper_bound}. However, one minor issue is that such lower bound cannot generally exist on all of $[0, \parti_1]$. To circumvent this, we notice that $p(\nu)/p(\nu_0)$ becomes small if $|\nu - \nu_0|$ is sufficiently large compared to $\sqrt{\parti_1}$. This suggests that we may focus on the range $|\nu - \nu_{0}| \leq n^{\gamma}\sqrt{\parti_1}$, where $\gamma$ is chosen as $\gamma = \frac{\beta}{2}-\frac{1}{4} = \frac{1}{8} $. And in this range, we want to obtain a gaussian lower bound of $p$. Focusing on the second derivative of $\log p(\nu)$ as before, we obtain
    \begin{align*}
        (\log p(\nu))''
        &= -\left( \frac{1}{\nu} + \cO\left(\frac{1}{\nu^2}\right) + \frac{1}{\parti_1 - \nu} + \cO\left(\frac{1}{(\parti_1 - \nu)^2}\right) \right) + \cO\left( \frac{\parti_1}{a^2} \right) \\
        &= -\frac{\parti_1}{\nu( \parti_1 - \nu)} + \cO\left( n^{-2\beta} \right) + \cO\left( n^{-3/2} \right),
    \end{align*}
    where both estimates $\sum_{n=1}^{\infty} \frac{1}{(n+x)^2} = \frac{1}{x} + \cO\left(\frac{1}{x^2}\right)$ uniformly in $x > 0$ and $\left|\frac{1}{a-\nu+k}\right| \leq \frac{2}{a}$ are exploited in the first step. To simplify further, we note that
    \begin{align*}
        \left| \nu - \frac{\parti_1}{2} \right|
        \leq \left|\nu - \nu_0 \right| + \left| \nu_0 - \frac{\parti_1}{2} \right|
        \leq n^{\gamma}\sqrt{\parti_1} + \cO\left( \frac{\parti_1^2}{a} \right)
        \leq \cO \left( \frac{\parti_1}{n^{1/4}} \right).
    \end{align*}
    In the last step, we made use of the bounds $\parti_1/a = \cO(n^{-1/4})$ and $n^{\gamma}/\sqrt{\parti_1} \leq n^{\gamma-\beta/2} = n^{-1/4}$. So it follows that
    \begin{align*}
        \frac{\parti_1}{\nu( \parti_1 - \nu)}
        = \frac{4}{\parti_1} \cdot \frac{1}{1 - \left( \frac{\nu - (\parti_1/2)}{\parti_1/2} \right)^2}
        = \frac{4}{\parti_1} \left(1 + \cO\left( n^{-1/2} \right) \right)
        = \frac{4}{\parti_1} + \cO\left( n^{-\beta-1/2} \right).
    \end{align*}
    Plugging this into the asymptotic formula of $(\log p(\nu))''$ and combining all the error terms into a single one, we end up with
    \begin{align*}
        (\log p(\nu))''
        &= -\frac{4}{\parti_1} + \cO\left( n^{-5/4} \right).
    \end{align*}
    Given this asymptotic formula, we can proceed as in \eqref{eq:pnu_upper_bound} to obtain
    \begin{align*}
        p(\nu)
        &= p(\nu_0) \exp \left\{ -\frac{2}{\parti_1}(\nu - \nu_{0})^2 + \cO\left( n^{2\gamma-5/4} \right) \right\}.
    \end{align*}
    From this, we have
    \begin{align*}
        \sum_{\nu : |\nu - \nu_0| \leq n^{\gamma}\sqrt{\parti_1}} \frac{p(\nu)}{p(\nu_0)}
        &= e^{\cO(n^{-1/4})} \int_{|t| \leq n^{\gamma}\sqrt{\parti_1}} e^{-\frac{2}{\parti_1}t^2} \, \mathrm{d}t \\
        &= e^{\cO\left( n^{-1/4} \right)} \left( \int_{\bR} e^{-\frac{2}{\parti_1}t^2} \, \mathrm{d}t - \int_{|t| > n^{\gamma}\sqrt{\parti_1}} e^{-\frac{2}{\parti_1}t^2} \, \mathrm{d}t \right) \\
        &= e^{\cO\left( n^{-1/4} \right)} \sqrt{\frac{\pi \parti_1}{2}} + \cO\left( e^{-2n^{\gamma}} \right)
    \end{align*}
    The first step follows by noting that $-\frac{2}{\parti_1}(t - \nu_0)^2 = -\frac{2}{\parti_1} (\nu - \nu_0)^2 + \cO\left(n^{\gamma-\beta/2}\right) $ if $|t - \nu| \leq 1$ and~$\gamma - \beta/2 = -1/4$. Also, in the last step, we utilized the tail estimate $\int_{x}^{\infty} e^{-t^2/2} \, \mathrm{d}t < e^{-x^2/2}/x$, which is valid for $x > 0$, to produce a stretched-exponential decay. Similar reasoning shows that
    \begin{align*}
        \sum_{\nu : |\nu - \nu_0| \leq n^{\gamma}\sqrt{\parti_1}} \frac{p(\nu)}{p(\nu_0)}
        \leq \cO \left( \int_{|t| > n^{\gamma}\sqrt{\parti_1}} e^{-\frac{2}{\parti_1+2}t^2} \, \mathrm{d}t \right)
        \leq \cO\left( e^{-2n^{\gamma}} \right).
    \end{align*}
    Putting all the estimates altogether, we obtain
    \begin{align}
        \label{eq:ka1_est_1}
        \mK_{a,1}
        = \sqrt{\frac{\pi \parti_1}{2}} e^{\cO\left( n^{-1/4} \right)} p(\nu_{0}).
    \end{align}
    
    In view of \eqref{eq:ka1_est_1}, it remains to estimate $p(\nu_0)$. Since $\nu_{0} - \tilde{\nu}_0 = \cO(1)$, it follows $\nu_{0}$ satisfies the same asymptotic formula as in \eqref{eq:nutilde0_asymp}. Write $\mu = \nu_{0} - \frac{\parti_1}{2}$. We know that $\mu = o(\parti_1)$, or more precisely, $\mu/\parti_1 = \cO(n^{-1/4})$. Then, by using Stirling's approximation \cite{Robbins}
    \begin{align*}
        \log(n!)
        = \left(n + \frac{1}{2}\right)\log n - n + \log\sqrt{2\pi} + \cO\left(\frac{1}{n}\right),
    \end{align*}
    we obtain
    \begin{align*}
        \log \left[ \frac{1}{2^{\parti_1}} \binom{\parti_1}{\nu_{0}} \right]
        &= -\parti_1 \log 2 + \log (\parti_1 !) - \log \left( \frac{\parti_1}{2} + \mu \right)! - \log \left( \frac{\parti_1}{2} - \mu \right)! \\
        %
        %
        &=  - \left(\frac{\parti_1}{2} + \mu + \frac{1}{2}\right)\log \left( 1 + \frac{2\mu}{\parti_1} \right)
        - \left(\frac{\parti_1}{2} - \mu + \frac{1}{2}\right)\log \left( 1 - \frac{2\mu}{\parti_1} \right) \\
        & \qquad + \log 2 - \frac{1}{2}\log \parti_1 - \log\sqrt{2\pi} + \cO\left(\frac{1}{\parti_1}\right) \\
        &= \frac{\parti_1}{2} \left[ \left( \frac{1}{\parti_1} - 1 \right) \left( \frac{2\mu}{\parti_1} \right)^2 + \left( \frac{1}{2\parti_1} - \frac{1}{6} \right) \left( \frac{2\mu}{\parti_1} \right)^4 + \cO \left( \frac{2\mu}{\parti_1} \right)^6 \right] \\
        &\qquad + \frac{1}{2}\log\left(\frac{2}{\pi \parti_1} \right) + \cO \left( n^{-\beta} \right).
    \end{align*}
    This can be further simplified by noting that $\frac{\mu}{\parti_1} = - \frac{\parti_1}{4a} + \frac{\parti_1^3}{16a^3} + \cO\left(\frac{1}{n^{1/4}\parti_1}\right) = \cO(n^{-1/4})$, and the result is
    \begin{align}
        \log \left[ \frac{1}{2^{\parti_1}} \binom{\parti_1}{\nu_{0}} \right]
        &= \frac{1}{2}\log\left(\frac{2}{\pi \parti_1} \right) - \frac{2\mu^2}{\parti_1} - \frac{4\mu^4}{3\parti_1^3} + \cO \left(n^{-1/4}\right) \nonumber \\
        &= \frac{1}{2}\log\left(\frac{2}{\pi \parti_1} \right) - \frac{\parti_1^3}{8a^2} + \frac{11\parti_1^5}{192a^4} + \cO \left(n^{-1/4}\right). \label{eq:ka1_est_2}
    \end{align}
    For the remaining factor, we estimate it as follows.
    \begin{align*}
        \log \left[ \frac{(a + \parti_1 - 1 - \nu_0)!}{(a - \nu_0)! a^{\parti_1 - 1}} \right]
        &= \log \left[ \frac{(a + \parti_1 - \nu_0)!}{(a - \nu_0)! a^{\parti_1}} \right] + \log \left[ \frac{a}{a + \parti_1 - \nu_0} \right] \\
        &= - \parti_1 + \left(a + \frac{1}{2} + \frac{\parti_1}{2} - \mu \right) \log \left( 1 + \frac{\frac{\parti_1}{2} - \mu}{a} \right) \\
        &\qquad - \left(a + \frac{1}{2} - \frac{\parti_1}{2} - \mu\right) \log \left( 1 - \frac{\frac{\parti_1}{2} + \mu}{a} \right) + \cO\left(n^{-1/4}\right)
    \end{align*}
    After some painful expansion, we end up with
    \begin{align}
        \label{eq:ka1_est_3}
        \log \left[ \frac{(a + \parti_1 - 1 - \nu_0)!}{(a - \nu_0)! a^{\parti_1 - 1}} \right]
        &= \frac{5 \parti_1^3}{24 a^2} - \frac{73 \parti_1^5}{960 a^4} + \cO\left(n^{-1/4}\right).
    \end{align}
    Therefore, the conclusion follows by combining \eqref{eq:ka1_est_1}, \eqref{eq:ka1_est_2} and \eqref{eq:ka1_est_3} altogether.
\end{proof}

\subsection{Estimation of the peak generating function}

\begin{lemma}
    \label{lemma:large_range}
    Let $\delta \in (0, \delta_0)$ and write $\alpha_1 = \parti_1 / n$ for the density of fixed points. Then
    \begin{align*}
        \mL_{(\delta n^{5/4}, \infty)} = \exp\left\{ \frac{\alpha_1^3}{3} s\sqrt{n} + \left( \frac{\alpha_1^3}{18} - \frac{3\alpha_1^5}{10} + \frac{2\alpha_1^6}{9} \right) s^2 + \cO\left(n^{-1/4}\right) \right\}
    \end{align*}
    holds in the range $\delta n^{5/4} \geq \max\{e^{c_1}, 2\}n$. Moreover, the implicit bound of the error term depends only on~$\delta$ and~$s$.
\end{lemma}

Following Kim and Lee's method \cite{KimLee}, we will utilize Laplace's method to approximate the sum by the integral of a certain gaussian density function and show that the relative error due to this approximation can be controlled in an explicit and uniform manner. The following simple lemma is useful for this purpose.

\begin{lemma}
    \label{lemma:approx_gaussian}
    Define $f_n: \bR \to \bR$ by $f_n(x) = \left( 1 + \frac{x}{\sqrt{n}} \right)^n e^{-\sqrt{n}x} \mathbf{1}_{[-\sqrt{n},\infty)}(x)$. Then
    \begin{enumerate}[topsep=0.25em,itemsep=0em,label={(\arabic*)}]
        \item If $x \geq 0$ and $l > n > 0$, then $f_l(x) \leq f_n(x) \leq (2/\sqrt{e})^n e^{-\sqrt{n}x/2}$.
        
        \item If $x \leq 0$ and $l > n > 0$, then $f_n(x) \leq f_l(x) \leq e^{-x^2/2}$.
        
        \item $f_n(x) \to e^{-x^2/2}$ pointwise as $n\to\infty$.
    \end{enumerate}
\end{lemma}

The estimation of $f_n$ is a recurring tool in previous works (see Lemma 4.3 of \cite{KimLee} and the proof therein, for instance) and requires only basic calculus computation. Nevertheless, we include the proof for self-containedness.

\begin{proof}
    Let $h(t, x) = t \log\left( 1 + \frac{x}{\sqrt{t}}\right) - \sqrt{t}x$. It is easy to check that
    \begin{itemize}[topsep=0.25em,itemsep=0em]
        \item $x \mapsto h(t, x)$ is concave on $(0, \infty)$ for each $t \in (0, \infty)$,
        \item $t \mapsto h(t, x)$ is decreasing on $(0, \infty)$ for each $x \geq 0$,
        \item $t \mapsto h(t, x)$ is increasing on $(x^2, \infty)$ for each $x \leq 0$, and
        \item $h(t, x) \to -x^2/2$ as $t\to\infty$ for each $x \in \bR$.
    \end{itemize}
    From $f_n(x) = e^{h(n, x)}$, the assertions (2) and (3) follows immediately. Moreover, we may exploit the concavity of $x \mapsto h(t, x)$ to bound $h(t, x) \leq h(t, \sqrt{n}) + \frac{\partial h}{\partial x}(t, \sqrt{n})(x - \sqrt{n})$, which gives (1).
\end{proof}

Now we return to the proof of the main claim of this section.

\begin{proof}[Proof of Lemma \ref{lemma:large_range}]
    Assume that $\delta n^{5/4} \geq \max\{e^{c_1}, 2\}n$ holds. Then, by Lemmas \ref{lemma:fai_est}, \ref{lemma:kai_asymp_large}, and \ref{lemma:ka1_asymp_large}, we have
    \begin{align*}
        \mL_{(\delta n^{5/4}, \infty)}
        = e^{\cO(n^{-1/4})} \frac{\log^{n+1}(1/t)}{n!} \sum_{a > \delta n^{5/4}} t^a a^n \exp\left\{ \frac{\parti_1^3}{12a^2} - \frac{3\parti_1^5}{160a^4} \right\}
    \end{align*}
    Next, we approximate the sum in the right-hand side by its integral analogue. If $x \in \bR$ and $a > \delta n^{5/4}$ are such that $|x - a| \leq 1$, then
    \begin{itemize}
        \item $t^x = t^{a}e^{\cO(\log t)} = t^a e^{\cO(n^{-1/4})}$,
        \item $x^n = a^n e^{n\log(x/a)} = a^n e^{\cO(n/a)} = a^n e^{\cO(n^{-1/4})}$, \text{ and}
        \item for each $k \geq 0$ given, $\frac{\parti_1^{k+1}}{x^k} = \frac{\parti_1^{k+1}}{a^k}\left(1 + \cO\left(\frac{1}{a}\right)\right)^k = \frac{\parti_1^{k+1}}{a^k} + \cO\left(\frac{\parti_1}{a}\right)^{k+1} = \frac{\parti_1^{k+1}}{a^k} + \cO(n^{-1/4})$. The implicit error bound now depends on $k$ as well. However, it will be used only for $k = 2$ and $k =4$, and so, this causes no harm for our objective of retaining error bounds depending only on $s$ and $\delta$.
    \end{itemize}
    This allows us to approximate the sum by its integral analogue at the expense of the relative error~$e^{\cO(n^{-1/4})}$, yielding
    \begin{align}
        \label{eq:integral_to_est}
        \mL_{(\delta n^{5/4}, \infty)}
        = e^{\cO(n^{-1/4})} \frac{\log^{n+1}(1/t)}{n!} \mJ, \qquad \text{where }
        \mJ = \int_{\delta n^{5/4}}^{\infty} t^x x^n \exp\left\{ \frac{\parti_1^3}{12x^2} - \frac{3\parti_1^5}{160x^4} \right\} \, \mathrm{d}x
    \end{align}
    So it remains to estimate $\mJ$. To this end, we substitute $x = \frac{n}{\log (1/t)} \left( 1 + \frac{w}{\sqrt{n}} \right)$. For the sake of brevity, we also write $c_5 = c_5(n) = 1-\delta n^{1/4}\log(1/t)$. Although $c_5$ depends on $n$ and $s$, the choice of $\delta$ and \eqref{eq:delta} tell us that $c_5$ is uniformly away from $0$ and $1$, which will be sufficient for our purpose. Then,
    \begin{align*}
        \mJ
        &= \int_{-c_5\sqrt{n}}^{\infty} \exp\Bigg\{ \left( n \log \left( \frac{n}{\log(1/t)} \right) - n \right) + \left( n \log \left( 1 + \frac{w}{\sqrt{n}} \right) - \sqrt{n}w \right) \\
        &\hspace{8em} + \frac{\alpha_1^3 n \log^2(1/t)}{12\left(1+(w/\sqrt{n})\right)^2} - \frac{3\alpha_1^5 n \log^4(1/t)}{160\left( 1 + (w/\sqrt{n}) \right)^4} \Bigg\} \frac{\sqrt{n}}{\log(1/t)} \, \mathrm{d}w.
    \end{align*}
    The first two grouped terms in the exponent of the integrand are easily controlled, as they originated from the `unperturbed term' $t^x x^n$. So, it suffices to study the effect of the `perturbation terms'. Taking advantage of the explicit formula of the perturbation term, one may expand
    \begin{align*}
        \frac{\alpha_1^3 n \log^2(1/t)}{12\left(1+(w/\sqrt{n})\right)^2}
        &= \frac{\alpha_1^3 n \log^2(1/t)}{12} \left( 1 - \frac{(w/\sqrt{n})\left(2+(w/\sqrt{n})\right)}{\left(1+(w/\sqrt{n})\right)^2} \right)
    \end{align*}
    Plugging this back in, the integral takes the form
    \begin{align*}
        \mJ
        = \frac{n^{n+1/2}e^{-n}}{\log^{n+1} (1/t)} \exp\left\{ \frac{\alpha_1^3}{12} n \log^2(1/t) \right\} \int_{-c_5\sqrt{n}}^{\infty} f_n(w) e^{g_n(w)} \, \mathrm{d}w,
    \end{align*}
    where $f_n$ is as in Lemma \ref{lemma:approx_gaussian} and $g_n$ is defined by
    \begin{align*}
        g_n(w) = - \frac{\alpha_1^3 \sqrt{n}\log^2(1/t) w(2 + (w/\sqrt{n}))}{12(1 + (w/\sqrt{n}))^2} - \frac{3\alpha_1^5 n \log^4(1/t)}{160\left( 1 + (w/\sqrt{n}) \right)^4}.
    \end{align*}
    As mentioned before, $c_5$ is uniformly away from $1$, meaning that $\sup_{n \geq 1} c_5(n) < 1$ holds. Then $g_n(w) \leq 0$ for $w \geq 0$ and $g_n(w) \leq -c_6 w$ for $w \in [-c_5\sqrt{n}, 0]$, where $c_6 > 0$ is a constant depending only on $s$. Now using the tail estimates in Lemma \ref{lemma:approx_gaussian}, we can check that
    \begin{align*}
        \int\limits_{\substack{w \geq -c_4\sqrt{n} \\ |w|\geq \log n} } f_n(w) e^{g_n(w)} \, \mathrm{d}w
        = \cO\left(n^{-1/4}\right).
    \end{align*}
    Moreover, if $|w| \leq \log n$, then using $\log(1/t) = \frac{2\sqrt{s}}{n^{1/4}} + \frac{s^{3/2}}{6n^{3/4}} + \cO\left(n^{-5/4}\right)$,
    \begin{align*}
        f_n(w) = -\frac{w^2}{2} + \cO\left(\frac{\log^3 n}{\sqrt{n}}\right), \qquad
        g_n(w) = - \frac{\alpha_1^3}{3} s w - \frac{3\alpha_1^5}{10} s^2 + \cO\left(\frac{\log n}{\sqrt{n}} \right).
    \end{align*}
    Plugging this back to $\mL_{(\delta n^{5/4}, \infty)}$ and utilizing Stirling's formula,
    \begin{align*}
        \mL_{(\delta n^{5/4}, \infty)}
        &= \frac{1}{\sqrt{2\pi}} \exp\left\{ \frac{\alpha_1^3}{3} s\sqrt{n} + \left(\frac{\alpha_1^3}{18} - \frac{3\alpha_1^5}{10}\right) s^2 + \cO(n^{-1/4}) \right\} \\
        &\hspace{4em} \times \left( \int_{|w|\leq\log n} \exp\left\{ -\frac{w^2}{2}- \frac{\alpha_1^3}{3} s w \right\} \, \mathrm{d}w + \cO\left(n^{-1/4}\right) \right) \\
        &= \exp\left\{ \frac{\alpha_1^3}{3}s\sqrt{n} + \left( \frac{\alpha_1^3}{18} - \frac{3\alpha_1^5}{10} + \frac{2\alpha_1^6}{9} \right)s^2 + \cO\left(n^{-1/4}\right) \right\}
    \end{align*}
    as required.
\end{proof}

With all the ingredients ready, we immediately obtain the proof of Theorem \ref{maintechthm}.

\begin{proof}[Proof of Theorem \ref{maintechthm}]
    In the proof of Theorem \ref{thm:clt_peaks_Sn}, we checked that
    \begin{align*}
        \left( \frac{2(1-t)}{(1+t)\log(1/t)} \right)^{n+1} = \exp\left\{ -\frac{s}{3}\sqrt{n} + \frac{1}{45}s^2 + \cO(n^{-1/4}) \right\}.
    \end{align*}
    Moreover, if we fix $\delta \in (0, \delta_0)$, by Lemma \ref{lemma:small_range}, we can choose $\rho \in (0, 1)$, independent of $n$ and $\lambda$, so that $\mL_{[1,\delta n^{5/4}]} = \cO(\rho^{n})$. Also, if $n$ is sufficiently large so that $\delta n^{5/4} \geq \max\{e^{c_1}, 2\}n$, Lemma \ref{lemma:large_range} gives a uniform estimate on $\mL_{(\delta n^{5/4}, \infty)}$. Finally, if $\pi$ is chosen uniformly at random from $\cC_{\lambda}$, then
    \begin{align*}
        \E\left[ e^{-s p(\pi) / \sqrt{n}} \right]
        = e^{s/\sqrt{n}} \left( \frac{2(1-t)}{(1+t)\log(1/t)} \right)^{n+1} \mL_{[1, \infty)}.
    \end{align*}
    Plugging in all the estimates and taking advantage of the fact that $\mL_{[1,\delta n^{5/4}]} = \cO(\rho^n)$ can be absorbed into the relative error $\cO(n^{-1/4})$, we have
    \begin{align*}
        \E\left[ e^{-s p(\pi) / \sqrt{n}} \right]
        = \exp\left\{ \left( -\frac{s}{3}\sqrt{n} + \frac{1}{45}s^2 \right) + \left( \frac{\alpha_1^3}{3} s\sqrt{n} + \left( \frac{\alpha_1^3}{18} - \frac{3\alpha_1^5}{10} + \frac{2\alpha_1^6}{9} \right) s^2 \right) + \cO\left(n^{-1/4}\right) \right\}.
    \end{align*}
    This provides the desired bound for the term $E_{\lambda,s}$ appearing in the statement of Theorem \ref{maintechthm}, completing the proof.
\end{proof}

\section*{Acknowledgement}

Fulman was supported by Simons Foundation Grant 400528.

\end{document}